\documentclass[12pt,reqno]{amsart}   	
\usepackage[letterpaper, margin=1in]{geometry}
\usepackage{amssymb}
\usepackage{amsmath}
\usepackage{amsfonts}
\usepackage[utf8]{inputenc}
\usepackage{amsthm}
\usepackage{tikz-cd}
\usepackage{array}

\newtheorem{thm}{Theorem}[section]
\newtheorem{lemma}[thm]{Lemma}
\newtheorem{df}[thm]{Definition}

\newtheorem{ex}[thm]{Example}

\newtheoremstyle{remark}
    {\dimexpr\topsep/2\relax} 
    {\dimexpr\topsep/2\relax} 
    {}          
    {}          
    {\bfseries} 
    {.}         
    {.5em}      
    {}          

\theoremstyle{remark}
\newtheorem{remark}[thm]{Remark}

\newcommand{\Z}{\mathbb{Z}}

\newcommand{\R}{\mathbb{R}}

\newcommand{\N}{\mathbb{N}}

\title{A note on the real Jacobian conjecture in degree 7}
\author{Tomasz  Kowalczyk}
\date{}

\begin{document}

\keywords{real Jacobian conjecture, Newton polygon}
\subjclass[2020]{14R15, 14P99}
\maketitle

\begin{abstract}
Let $(p,q)$ be a Jacobian pair. We show that the real Jacobian conjecture holds if the degree of $p$ is 7 and the highest degree homogenous part is of the form $\alpha x^7 + \beta x^6y$ for $\alpha^2+\beta^2 \neq 0$. We then show that there are no atypical Jacobian pairs such that $\deg p =7$ and $\deg q$ is even and coprime with 7.
\end{abstract}

\section*{Introduction}
Let $(p,q):\R^2 \rightarrow \R^2$ be a polynomial map. We say that $(p,q)$ is a Jacobian pair if its Jacobian $J(p,q)$ does not vanish. The real Jacobian conjecture states that under the above assumption, $(p,q)$ is an injective mapping.
If a Jacobian pair $(p,q)$ satisfies the real Jacobian conjecture, we say that it is a typical Jacobian pair, otherwise we will call it an atypical Jacobian pair.

It is known since the work of Pinchuk \cite{pinchuk1994} that atypical Jacobian pairs exists. However some questions remains open. What are the degree conditions of $p$ and $q$ such that there may exist an atypical Jacobian pair? What is the smallest possible degree of an atypical Jacobian pair? Both questions were extensively studied. Regarding the first one, it is known that if the degree of $p$ is at most 6, then $(p,q)$ is a typical Jacobian pair \cite{braunEtAl2026}, see also \cite{braunsantosteixeira2022, braunfernandesOO2025, braunorefice2016,gwozdziewicz2001}. This result was a sequence of papers across almost 30 years. Regarding the second question, it was recently proved that there exists an atypical Jacobian pair with $p$ of degree 7 and $q$ of degree 29. And so, both questions are almost fully resolved. Again, this is obtained by a series of papers that focused on producing counterexamples of smaller and smaller degree (cf. \cite{braunFernandes2023,campbell2011, fernandes2022}). Let us also mention papers which provide sufficient conditions on $p$ and $q$ different, than the degrees \cite{braunfernandes2025,dominguesllibremello2026,lietian2025,tiancen2024}.


Structure of the paper is as follows. We start with some preliminary results. Section 2 contains the study of Jacobian pairs $(p,q)$ with $p$ of degree 7 and highest homogenous part equal to $\alpha x^7 + \beta x^6y$. In the last section we use those results to show that if $\deg f =7$ and the degree of $q$ is even and coprime with 7, then the pair $(p,q)$ is typical.

Our approach combines ideas and results from \cite{braunEtAl2026,gwozdziewicz2025}.

\section{Preliminaries}
Let $p(x,y)=\sum_{(i,j)\in\N^2} a_{ij} x^i y^j  \in \R[x,y]$ be a polynomial. Support of $p$, denoted by $\mathrm{supp}(p)$ is the set all pairs $(i,j) \in \N^2$ such that $a_{ij} \neq 0$. We define Newton polygon of $p$ to be the set
$$ \Delta_p= \mathrm{conv}( \{ (i,j)\in \N^2 : (i,j)\in \mathrm{supp}(p) \}) \subset \R^2 .$$
In the above, $\mathrm{conv}(A)$ is the convex hull of the set $A$. In order to work with the Newton polygon more easily we have to define several more notions.

Let now $\Delta \subset \R^2$ be a compact set and $\xi \in \R^2$ be a non-zero vector. We define 
$l(\Delta, \xi) = \max \{\langle \alpha,\xi \rangle : 
\alpha \in \Delta\}$ and
$\Delta^\xi = \{ \alpha \in \Delta : \langle \alpha,\xi \rangle =l(\Delta, \xi) \}$.
Note that if $\xi \in \R^2$ is a non-zero vector then $\Delta_p^\xi$ is either a vertex or an edge of the Newton polygon of $p$. Importance of these objects is shown in the following results

\begin{lemma}\cite[Corollary 2]{gwozdziewicz2001}\label{Delta}
    Let $p,q \in \R[x,y]$ and $\xi \in \R^2$ be a non-zero vector. Assume that $\Delta_p^\xi = \{\alpha \}$, $\Delta_q^\xi = \{ \beta\}$ and $\alpha, \beta$ are linearly independent. If $\alpha+\beta$ has even coordinate then $J(p,q)$ changes sign.
\end{lemma}

Let $E\subset \R^2$ be any set. We call the polynomial $p|_E=\sum_{(i,j)\in E} a_{ij} x^i y^j$ the symbolic restriction of $p$ to $E$. 

We say that a non-zero polynomial $p\in \R[x,y]$ is degenerate on an edge $S$ of its Newton polygon if $p|_E$ posses a multiple factor coprime with $xy$. In particular, if $p$ is degenerated on $S$, then $S$ has at least one interior lattice point.

We say that a non-zero polynomial $p$ is convenient if for some positive integers $k,m$ we have $(k,0) $ and $(0,m)$ both belong to $\mathrm{supp}(p)$. Equivalently, Newton polygon $\Delta_p$ touches both axis in points different from the origin.

Let $\Delta_p$ be a convenient Newton polygon of a polynomial $p$ and $S$ be an edge of $\Delta_p$. Let $\xi$ be a vector with coprime coordinates which points outward of $S$ and is orthogonal to $S$. If $\xi$  has at least one positive coordinate, then $S$ will be called an outer edge and $\xi$ will be called and outer normal vector.


We may now combine \cite[Lemma 3.4]{braunorefice2016} and \cite[Lemma 1]{gwozdziewicz2001} into

\begin{lemma}\label{nondegenerate on edges}
    Let $(p,q)$ be a Jacobian pair. Assume that $p$ is convenient and it is nondegenerate on each outer edge of $\Delta_p$. Then $(p,q)$ is a typical Jacobian pair.
\end{lemma}

\begin{ex}
    In the Figure \ref{example} we see a  Newton polygon of $p=9x^2y^4-6x^3y^2+x^4+y^4-2y^2-xy+x^3$. Edges marked in red are the edges having interior lattice points. Edges $S$ and $R$ are outer edges of $\Delta_p$. Polynomial $p$ is degenerate on $R$ since $p|_R=x^2(x-3y^2)^2$ and it is nondegenerate on the other two red edges. Clearly, $p$ is convenient.

\begin{figure}
    \centering
   
\begin{center}
\begin{tikzpicture}[scale=1.05]
\draw[->,>=latex,line width=0.7pt] (-0.4,0) -- (5.4,0) node[right]{$i$};
\draw[->,>=latex,line width=0.7pt] (0,-0.4) -- (0,5.4) node[above]{$j$};
\foreach \i in {0,...,5}{\foreach \j in {0,...,5}{\fill[black!35] (\i,\j) circle (1.7pt);}}
\fill[blue!8] (2,4) -- (0,4) -- (0,2) -- (1,1) -- (3,0) -- (4,0) -- cycle;
\draw[line width=1.6pt,red] (2,4) -- (0,4);
\draw[line width=1.6pt,red] (0,4) -- (0,2);
\draw[line width=1.6pt,blue] (0,2) -- (1,1);
\draw[line width=1.6pt,blue] (1,1) -- (3,0);
\draw[line width=1.6pt,blue] (3,0) -- (4,0);
\draw[line width=1.6pt,red] (4,0) -- (2,4);
\fill[black] (2,4) circle (3.4pt);
\fill[black] (3,2) circle (3.4pt);
\fill[black] (4,0) circle (3.4pt);
\fill[black] (0,4) circle (3.4pt);
\fill[black] (0,2) circle (3.4pt);
\fill[black] (1,1) circle (3.4pt);
\fill[black] (3,0) circle (3.4pt);
\node[above right,font=\scriptsize] at (2,4) {$(2,4)$};
\node[right,xshift=3pt,font=\scriptsize] at (3,2) {$(3,2)$};
\node[below,yshift=-9pt,xshift=4pt,font=\scriptsize] at (4,0) {$(4,0)$};
\node[above left,font=\scriptsize] at (0,4) {$(0,4)$};
\node[left,font=\scriptsize] at (0,2) {$(0,2)$};
\node[below left,font=\scriptsize] at (1,1) {$(1,1)$};
\node[below,yshift=-9pt,font=\scriptsize] at (3,0) {$(3,0)$};
\node[black,font=\bfseries\large] at (1,4.45) {$S$};
\node[black,font=\bfseries\large] at (3.42,2.5) {$R$};
\end{tikzpicture}
\end{center}
 \caption{Newton polygon of $p=9x^2y^4-6x^3y^2+x^4+y^4-2y^2-xy+x^3$
}
    \label{example}
\end{figure}
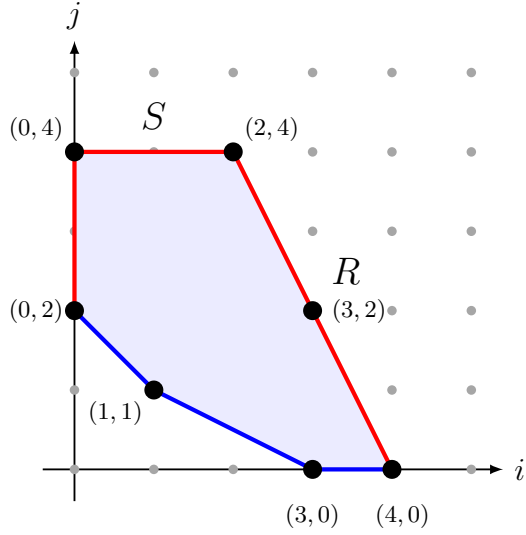

\end{ex}


\section{Jacobian pairs with $\deg p =7$ and $p_+=\alpha x^7+\beta x^6y$.}
    
Let us now discuss branches at infinity of $p(x,y)=0$. By infinity, we will mean the boundary of $\R^2$ inside $\R\mathbb{P}^2$. Any branch $\gamma$ at the infinity (See \cite[Section 2]{braunEtAl2026}) can be put back to $\R^2$ with a parametrization $\gamma(t)=(x(t),y(t))$ for $0<|t|\leq \epsilon$ for some $\epsilon>0$ satisfying
\begin{equation}\label{parametrization}
\begin{aligned}
& x(t)=At^{-\alpha} + \text{higher degree terms} \\
& y(t)=Bt^{-\beta} + \text{higher degree terms}
\end{aligned}
\end{equation}

for some $A,B \in \R$, $AB \neq 0$, $\alpha, \beta \in \Z$ and $\alpha >0$ or $\beta >0$.

\begin{df}\cite[Definition 2.2]{braunEtAl2026}
Let $p(x,y)$ be a polynomial. We say that a branch $\gamma$ of $p(x,y)=0$ at infinity is associated with the outer edge $S$ of $\Delta_p$ if there exists a positive integer $m$ such that $(\alpha,\beta)=m\xi$, where $\alpha, \beta$ are given by (\ref{parametrization}) and $\xi$ is an outer normal vector of $S$.
\end{df}
If additionally, $p$ is convenient, then every branch at infinity is associated with an outer edge of $\Delta_p$ (cf. \cite[Remark 1]{braunEtAl2026}).

\begin{df}\cite[Definition 2.3]{braunEtAl2026}
    Let $p(x,y)$ be a convenient polynomial and $S$ be an outer edge of $\Delta_p$. We define $N_S(c)$ to be the number of branches of $p(x,y)-c=0$ at infinity associated with $S$. \\
    We define $N(c)$ to be total number of branches at infinity of $p(x,y)-c=0$.
\end{df}
The above definition is correct, since adding a constant to $p$ does not change the outer edges of $\Delta_p$.  If $p(x,y)$ is a convenient polynomial, then $N(c)=\sum N_S(c)$ where sum is taken over all outer edges of $\Delta_p$. Also, if $p$ is a submersion, then $N(c)$ is precisely the number of connected components of $p^{-1}(c)$ (cf. \cite[Remark 2]{braunEtAl2026}).

In the remainder of this section we will write $N$ and $N_S$ instead of $N(c)$ and $N_S(c)$ as we will be interested in properties of $N$ and $N_S$ as functions globally, and not at a single point.

\begin{lemma}\cite[Lemma 4.3]{braunEtAl2026}\label{Lemmma N-N_S}
Let $p(x,y)$ be a polynomial submersion with a convenient Newton polygon having an outer edge $S$ with exactly one interior lattice point. If the function $N-N_S$ is constant then $N \equiv 1$ or $p$ has no Jacobian mates.
    
\end{lemma}

\begin{remark}\label{N=1 then typical}
Let $(p,q)$ be Jacobian pair. If for $p$ the function $N$ is constant, then $N\equiv 1$. If additionally $p$ is convenient, then $(p,q)$ is a typical Jacobian pair (cf. \cite[Lemma 1.2]{braunorefice2016}).
\end{remark}

\begin{lemma}\cite[Lemma 4.4]{braunEtAl2026}\label{lemma 03-33}
    Let $p(x,y)$ be a polynomial submersion whose Newton polygon has an outer edge $S$ with endpoints $(0,3)$ and $(3,3)$.  Then:
    \begin{itemize}
        \item[i)] $N_S$ is constant and bigger than $0$, or
        \item[ii)] $p$ does not have a Jacobian mates, or
        \item[iii)] after an affine change of variables of the form $x:=x+a$, the Newton polygon of $p$ turns to one with edge with endpoints $(3,3)$ and $(2,3)$ and one with endpoints $(2,3)$ and $(0,1)$.
    \end{itemize}
\end{lemma}
For a polynomial $p(x,y)$ we denote by $p_+$ its highest homogenous part.
\begin{thm}\label{alpha}
    Let $(p,q):\R^2 \rightarrow \R^2$ be a Jacobian pair. If $p_+=\alpha x^7$, for some non-zero $\alpha \in \R$, then $(p,q)$ is a typical Jacobian pair. 
\begin{proof}
If $(p,q)$ is a Jacobian pair, then the composition with an affine map is again a Jacobian pair. As a consequence, we may assume that the Newton polygon of $p$ is convenient, has a non-zero constant term, and does not have an edge with a positive slope.

Assume by contrary that $(p,q)$ is an atypical Jacobian pair. Below, we list all possible Newton polygons for $p$. They are listed lexicographically. Edges without interior points are coloured with blue, and edges with interior points are coloured in red.

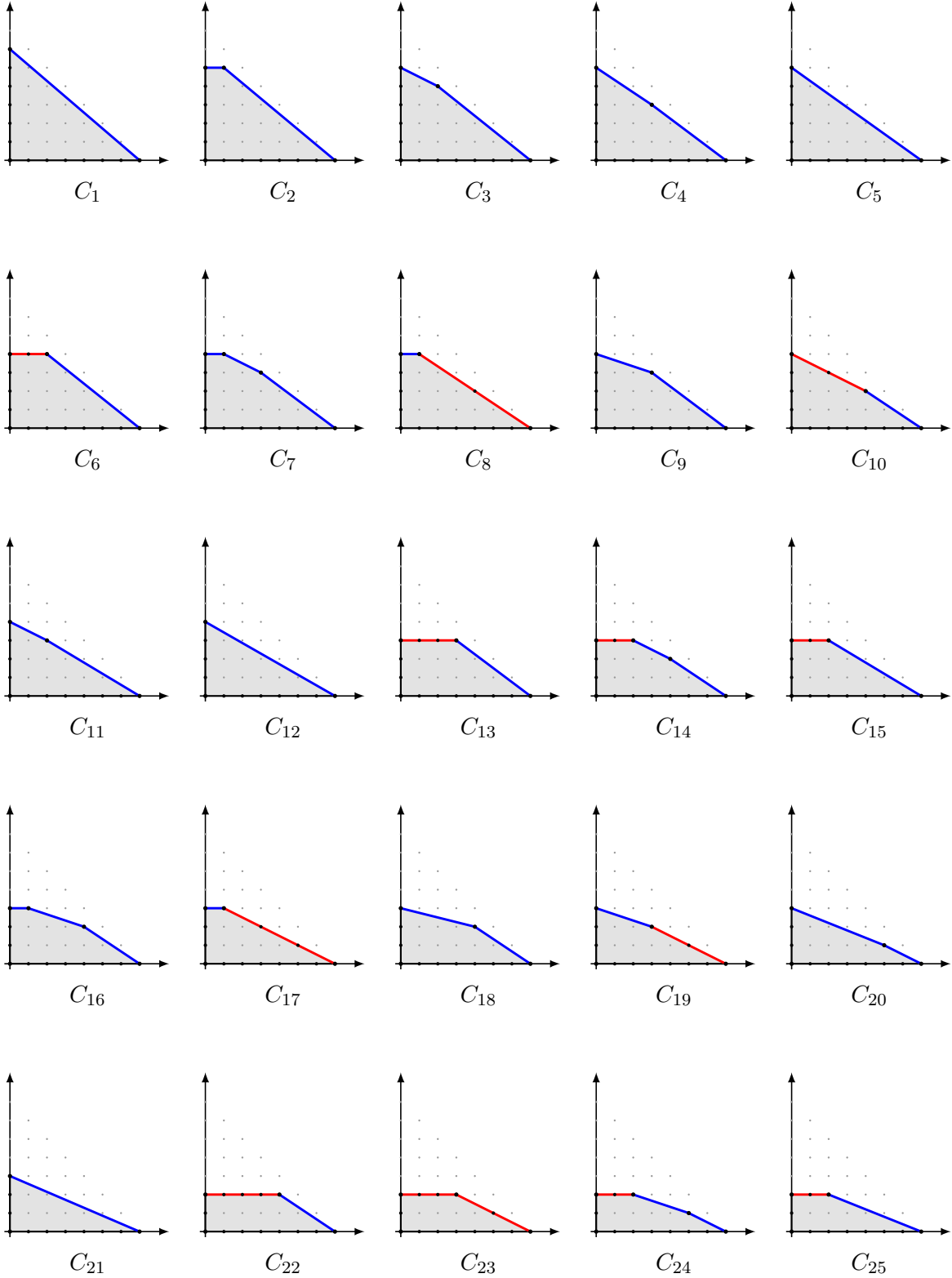
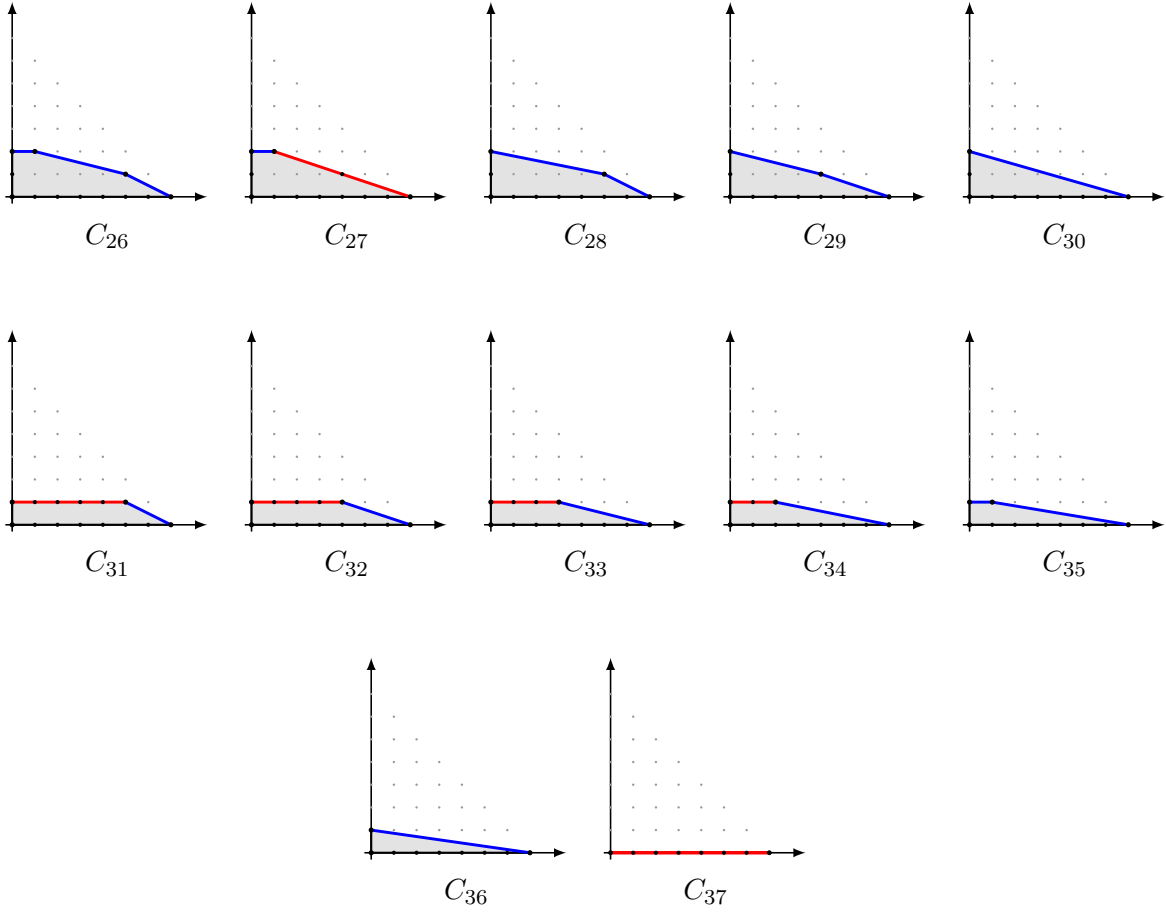
\begin{center}

\par\vspace{0.9cm}
\end{center}

Our aim is to check all possible cases and show that either $(p,q)$ is a typical Jacobian pair, or such $q$ does not exist (i.e. $p$ does not have a Jacobian mate).

If the $y$-degree of $p$ is at most 2, then \cite{braunFernandesMeza2025} implies that the pair $(p,q)$ is typical or $p$ does not have a Jacobian mate. As a consequences, Newton polygons $C_i$ for $i=22, 23, \dots, 37$ are done.

If the Newton polygon of $p$ does not have any degenerate edge, then the pair $(p,q)$ is typical by Lemma \ref{nondegenerate on edges}. Hence, Newton polygons $C_i$ for $i \in \{1,2,3,4,5, 7, 9,11,12, 16, 18,20 ,21 \}$ are done.

If $p$ is degenerate on only one edge, having only one interior lattice point, we may apply Lemma \ref{Lemmma N-N_S} with $S$ being the red edge $S$. Since $S$ is the only red edge, necessarily the function $N-N_S$ is constant, hence $N\equiv 1$ or $p$ does not have Jacobian mates. If $N \equiv 1$ then $(p,q)$ is a typical Jacobian pair by Remark \ref{N=1 then typical}. This solves the Newton polygons $C_i $ for $i \in \{6,8, 10, 14, 15,  19 \}$.

We are left with $C_{13}$ and $C_{17}$.
    Consider $C_{17}$. Denote the red edge by $S$. We have $p|_{S}=x(y-cx^2)^2(y+bx^2)$ for some $b,c \in \R$. We may now take
    the Jacobian pair $(\widetilde{p},\widetilde{q})=(p(x,y+cx^2),q(x,y+cx^2))$ which would be an atypical Jacobian pair with $\deg \widetilde{p}=6$, but such a case is impossible in view of the results of \cite{braunEtAl2026}.

    Consider $C_{13}$. We may now apply Lemma \ref{lemma 03-33}. If $(i)$ holds, then $N$ is constant, then by Lemma \ref{Lemmma N-N_S} and Remark \ref{N=1 then typical} either $p$ does not have a Jacobian mate, or $(p,q)$ is a typical Jacobian pair. If $(ii)$ holds, then obviously such $q$ does not exist. If $(iii)$ holds, we may again apply Lemma \ref{Lemmma N-N_S}.
    This finishes the proof.
\end{proof}
\end{thm}

In an analogous way we are able to consider the cases of the leading homogenous part equal to $p_+=\beta x^6y$ and $p_+=\alpha x^7 + \beta x^6y$ for $\alpha \beta \neq 0$. Proofs will be omitted, yet the possible Newton polygons will be present. The only polygons worth noting are $A_9, A_{11}, T_9, T_{11}$ where an appropriate substitution of the form $\widetilde{p}=p(x,y+cx^2)$ reduces the Newton polygon to either a polygon of a degree 6 polynomial, or a Newton polygon of a polynomial with highest homogenous part equal to $\alpha x^7$, but those were considered already.

\begin{thm}
    Let $(p,q):\R^2 \rightarrow \R^2$ be a Jacobian pair. If $p_+=\beta x^6y$ then $(p,q)$ is a typical Jacobian pair. 

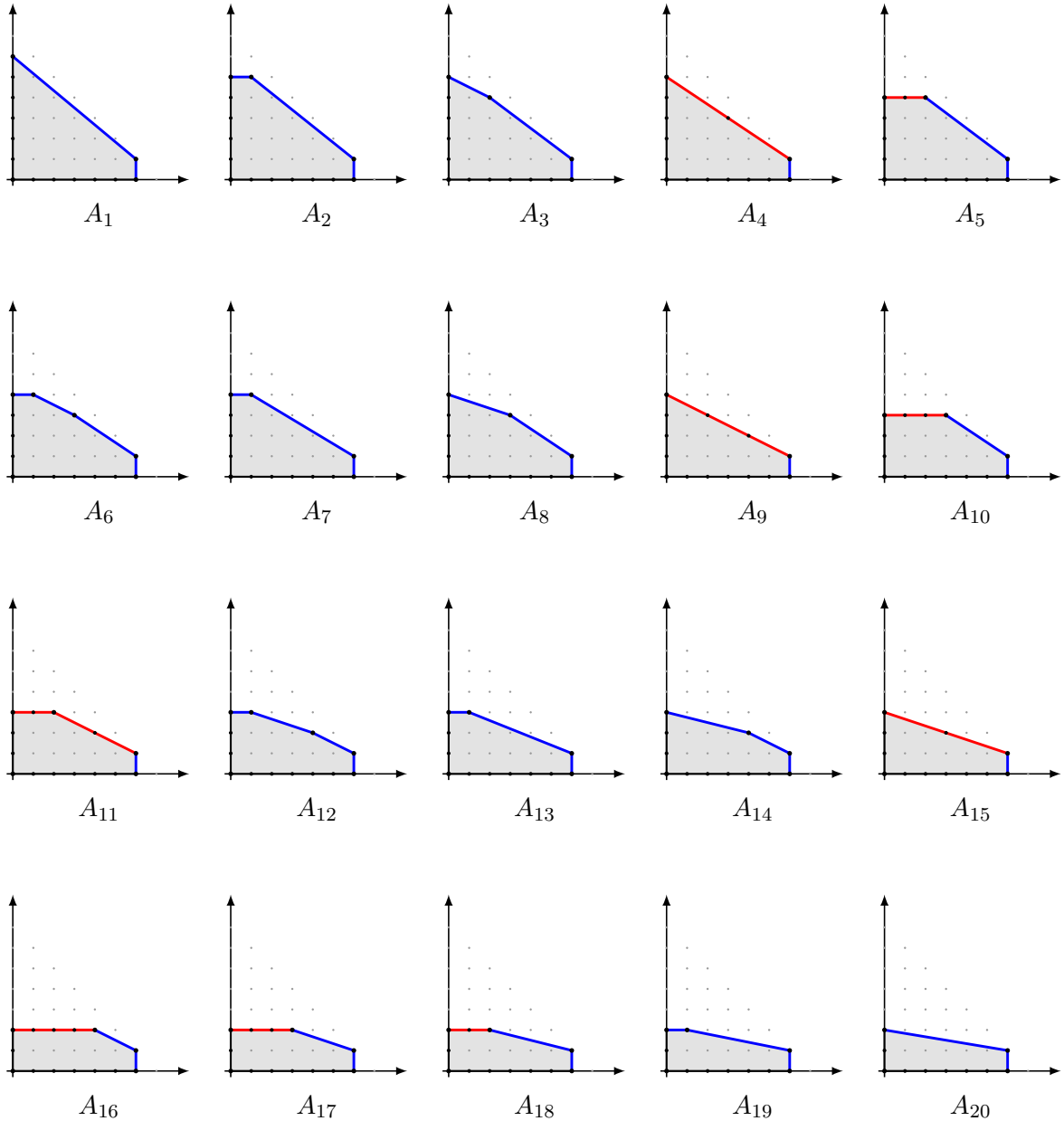
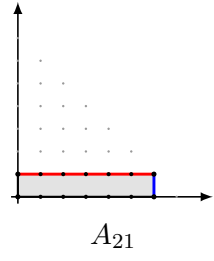
\begin{center}

\par\vspace{0.9cm}
\end{center}
    
\end{thm}

\begin{thm}
    Let $(p,q):\R^2 \rightarrow \R^2$ be a Jacobian pair. If $p_+=\alpha x^7+\beta x^6y$ for $\alpha\beta \neq 0$ then $(p,q)$ is a typical Jacobian pair.

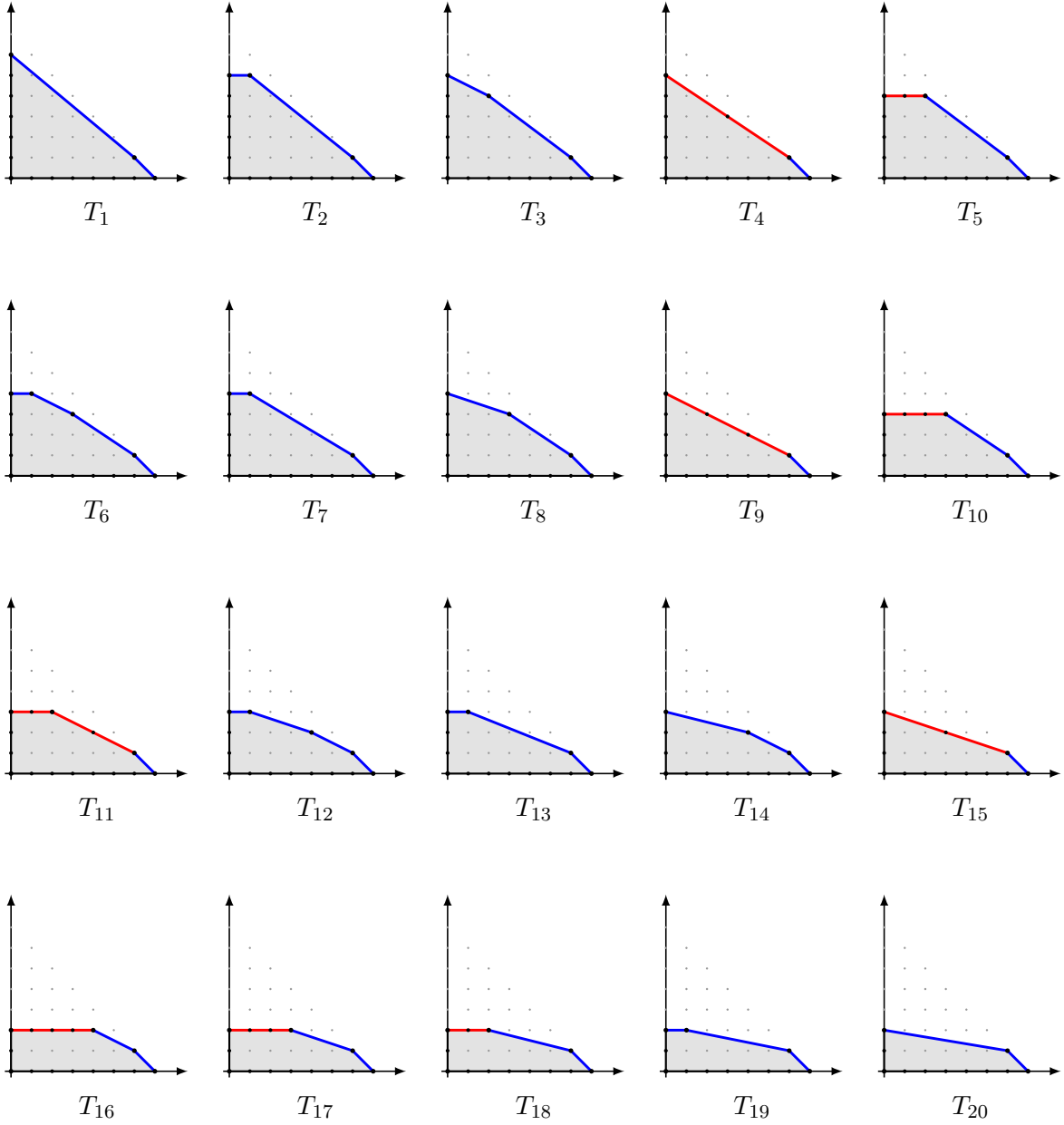
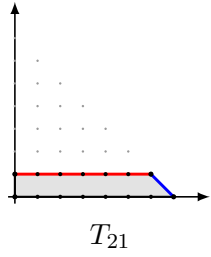
\begin{center}
%
\par\vspace{0.9cm}
\end{center}

\end{thm}

\section{Jacobian pairs with  $\deg p=7$ and $\deg q$ even, coprime with $7$.}
We may now generalize the result of Gwoździewicz \cite[Theorem 3.2]{gwozdziewicz2025}.

\begin{thm}\label{main}
    Let $(p,q)$ be a Jacobian pair with $\deg p=7$. Assume that the degree of $q$ is even and coprime with $7$. Then the Jacobian pair $(p,q)$ is typical.
 \end{thm}
\begin{proof}
    If the degree of $q$ is at most $6$ then $(p,q)$ is a typical Jacobian pair by \cite{braunEtAl2026}. 

    Assume that $d\geq 4$. Since the degree of $p$ is odd, (after a linear change of variables) we may assume that $x$ divides $p_+$. Let $k$ be the largest integer, such that $x^k$ divides $p_+$. If $k=7$ then $(p,q)$ is typical by Theorem \ref{alpha}. If $6\geq k \geq 1$, then $\alpha=(k,7-k)$ is a vertex of $\Delta_p$. Also, for any vector $\xi=[n, n+1]$ for $n\geq 7$ we have $\Delta_p^\xi=\{\alpha\}$. Let now $l$ be the smallest $i$ such that $(i, 2d-i) \in \mathrm{supp}(q)$. In particular, the point $(l, 2d-l)=\beta$ is a vertex of $\Delta_q$. Note that, for some vector $\xi=[n,n+1]$, with $n\geq 2d$, we have $\Delta_q^\xi=\{ \beta \}$. Now, $\alpha+\beta = (k+l, 2d+7-k-l)$ and precisely one coordinate is even. By the assumption $6\geq k \geq 1$ the points $\alpha$ and $\beta$ are linearly independent. Hence, by Lemma \ref{Delta} $(p,q)$ is not a Jacobian pair, since its Jacobian changes sign.
\end{proof}

\begin{remark}
    Both papers \cite{braunFernandes2023, braunEtAl2026} considered polynomial $p_7=x+y+3xy+3x^2y+2x^2y^2+3x^3y^2+x^4y^3$. In \cite{braunEtAl2026} they managed to find a Jacobian mate $q$, such that $(p_7.q)$ is an atypical Jacobian pair with $\deg q =29$. On the other hand, in \cite[Section 5]{braunFernandes2023} it was computationally checked that $p_7$ has no Jacobian mate of degree at most 15. Theorem \ref{main} implies that in order to determine smallest possible degrees of the counterexample to the real Jacobian conjecture, one has to resolve whether there are atypical Jacobian mates for $p_7$ of degrees $\{ 17,19,21,23,25,27,28 \}$.
\end{remark}


\section*{Acknowledgment}
The author certifies that the latex code for figures 1-4 was generated by AI. Newton polygons from figures 2-4 were previously drawn by hand and checked.
\bibliographystyle{plain}
\bibliography{refs.bib}

\begin{small}

\vspace{5pt}

\noindent
Tomasz Kowalczyk

\noindent
Institute of Mathematics

\noindent
Faculty of Mathematics and Computer Science

\noindent
Jagiellonian University

\noindent
ul. Łojasiewicza 6, 30-348 Kraków, Poland

\noindent
e-mail: tomek.kowalczyk@uj.edu.pl

\end{small}

\end{document}